\documentclass{amsart}
\usepackage{mathrsfs}
\usepackage{graphicx}

\usepackage{amssymb}
\usepackage{amsmath}
\usepackage{amsthm}
\usepackage{galois}
\usepackage[all,cmtip]{xy}

\newtheorem{thm}{Theorem}[section]
\newtheorem{cor}[thm]{Corollary}
\newtheorem{lem}[thm]{Lemma}
\newtheorem{prop}[thm]{Proposition}

\newtheorem{conj}[thm]{Conjecture}
\theoremstyle{definition}
\newtheorem{defn}[thm]{Definition}

\theoremstyle{remark}
\newtheorem{rem}[thm]{Remark}
\numberwithin{equation}{section}

\newcommand{\Z}{\mathbb Z}

\newcommand{\fix}{\mathrm{Fix}\,}
\newcommand{\eq}{\mathrm{Eq}\,}

\newcommand{\rk}{\mathrm{rk}}

\newcommand{\aut}{\mathrm{Aut}}
\newcommand{\edo}{\mathrm{End}}
\newcommand{\homo}{\mathrm{Hom}}

\newcommand{\cen}{\mathrm{Cen}}
\newcommand{\supp}{\mathrm{supp}}

\newcommand{\B}{\mathcal{B}}

\begin{document}

\title{Fixed subgroups are compressed in surface groups}

\author{Qiang Zhang}
\address{School of Mathematics and Statistics, Xi'an Jiaotong University, Xi'an 710049,
CHINA}
\email{zhangq.math@mail.xjtu.edu.cn}

\author{Enric Ventura}
\address{Department of Applied Mathematics III, Universitat Polit\`ecnica de Catalunya,
Manresa, Barcelona 08242, CATALONIA} \email{enric.ventura@upc.edu}

\author{Jianchun Wu}
\address{Department of Mathematics, Soochow University, Suzhou 215006, CHINA}
\email{wujianchun@suda.edu.cn}




\begin{abstract}
For a compact surface $\Sigma$ (orientable or not, and with boundary or not) we show
that the fixed subgroup, $\fix\B$, of any family $\B$ of endomorphisms of
$\pi_1(\Sigma )$ is compressed in $\pi_1(\Sigma )$ i.e., $\rk(\fix\B)\leqslant \rk(H)$
for any subgroup $\fix\B\leqslant H\leqslant \pi_1(\Sigma )$. On the way, we give a
partial positive solution to the inertia conjecture, both for free and for surface
groups. We also investigate direct products, $G$, of finitely many free and surface
groups, and give a characterization of when $G$ satisfies that $\rk(\fix
\phi)\leqslant \rk(G)$ for every $\phi\in \aut(G)$.
\end{abstract}
\maketitle

\section{Introduction}

For a finitely generated group $G$, let $\rk(G)$ denote the rank (i.e., the minimal
number of the generators) of $G$. There are lots of research works in the literature
about ranks of groups and, in particular, about controlling the rank of the intersection
of subgroups of $G$ in terms of their own ranks. An interesting case is when $G=F_r$ is
a finitely generated free group (of rank $r$): more than half a century ago H.~Neumann,
see~\cite{N1,N2}, conjectured that, for any two finitely generated subgroups $A$ and $B$
of $G$,
 $$
\rk(A\cap B)-1 \leqslant (\rk(A)-1)(\rk(B) -1).
 $$
After several partial results in this direction (for example, G.~Tardos~\cite{T} showed
it assuming either $A$ or $B$ of rank 2), this was finally proved in full generality
independently by J.~Friedman~\cite{F}, and by I.~Mineyev~\cite{M} (see also a simplified
version by W.~Dicks~\cite{D}).

Before this celebrated result was proved it had been shown that, for some special
subgroups in some special situations, one could say even more about their intersections;
this specially applies to fixed subgroups. For the case of finitely generated free
groups, there is a lot of research concerning properties of the subgroups fixed by
automorphisms or endomorphisms; today, the main open problem in this direction is
Conjecture~\ref{conj} below. Some of these results had been translated into surface
groups, where the situation is similar.

In the present paper, we give a new partial result in the direction of
Conjecture~\ref{conj} for finitely generated free groups, we prove some new results for
the case of surface groups, and we study the same problems in the family of groups
obtained by finitely many direct products of free and surface groups (where, as far as
we know, nothing was previously investigated in this direction). In this larger context,
the situation is much reacher, and new algebraic phenomena show up.

Let us first establish some notation and review the main results known about fixed
subgroups in finitely generated free and surface groups.

For two groups $G$ and $H$, let us denote the set of homomorphisms from $G$ to $H$ by
$\homo(G,H)$. Also, let us denote the set of endomorphisms of $G$ by $\edo(G)$, and the
set of automorphisms of $G$ by $\aut(G)$. For an arbitrary family $\B\subseteq
\homo(G,H)$, the \emph{equalizer} of $\B$ is
 $$
\eq \B :=\{g\in G \mid \beta_1(g)=\beta_2(g),\,\, \forall \beta_1,\beta_2\in \B\}\leqslant
G.
 $$
Similarly, for an arbitrary family $\B\subseteq \edo(G)$, the \emph{fixed subgroup} of
$\B$ is
 $$
\fix \B :=\{g\in G \mid \phi(g)=g, \forall \phi\in \B\} =\cap_{\phi\in
\mathcal{B}} \fix\phi \leqslant G.
 $$
Note that, if $G$ is a subgroup of $H$ and $\B\subseteq \homo(G,H)$ contains the inclusion
of $G$ in $H$ then, $\eq\B =\fix\B$.

In~\cite{BH}, Bestvina--Handel solved the famous Scott's conjecture:

\begin{thm}[Bestvina--Handel, \cite{BH}]\label{BH}
For any automorphism $\phi$ of a finitely generated free group $F$,
 $$
\rk(\fix\phi) \leqslant \rk(F).
 $$
\end{thm}

Later, Dicks--Ventura~\cite{DV} introduced the notions of inertia and compressedness,
and proved the following stronger result.

\begin{defn}
Let $G$ be an arbitrary finitely generated group. A subgroup $H\leqslant G$ is
\emph{inert in $G$} if $\rk(K\cap H)\leqslant \rk(K)$ for every (finitely generated)
subgroup $K\leqslant G$. A subgroup $H\leqslant G$ is \emph{compressed in $G$} if
$\rk(H)\leqslant \rk(K)$ for every (finitely generated) subgroup $K$ with $H\leqslant
K\leqslant G$.
\end{defn}

Note that the family of inert subgroups of $G$ is closed under finite intersections
(and, in the case of a free group ambient $G=F_r$, even closed under arbitrary
intersections by a standard argument on a descending chain of subgroups,
see~\cite[Corollary~ I.4.13]{DV}). Clearly, if $A$ is inert in $G$, then it is
compressed in $G$ and, in particular, $\rk(A)\leqslant \rk(G)$. It is not known in
general whether the converse is true for free groups, see~\cite[Problem~1]{DV} (and
also~\cite{V} as the ``compressed-inert conjecture").

The main result in~\cite{DV} is the following:

\begin{thm}[Dicks--Ventura, \cite{DV}]\label{inj fixed subgp of free gp inert}
Let $F$ be a finitely generated free group, and let $\B\subseteq \edo(F)$ be a family of
injective endomorphisms of $F$. Then the fixed subgroup $\fix\B$  is inert in $F$.
\end{thm}

As far as we know, the same statement for endomorphisms is still an open problem,
conjectured by Dicks--Ventura in~\cite[Problem~2]{DV} (see also~\cite{V} as the
``inertia conjecture").

\begin{conj}\label{conj}
For an arbitrary family of endomorphisms $\B\subseteq \edo(F)$ of a finitely generated
free group $F$, $\fix\B$ is inert in $F$.
\end{conj}

Few progress has been done in this interesting direction during the last twenty years:
only two later results gave supporting evidence to this conjecture. The first one was
given by G.~Bergman in~\cite{B}, where he extended Bestvina--Handel's result to
arbitrary families of endomorphisms:

\begin{thm}[Bergman, \cite{B}]
Let $F$ be a finitely generated free group, and let $\B \subseteq \edo(F)$. Then, $\rk
(\fix\B)\leqslant\rk(F)$.
\end{thm}

The following result from the same paper will also be used in our arguments below
(see~\cite[Corollary 12]{B}).

\begin{thm}[Bergman~\cite{B}]\label{intersection of retracts of free gp}
Let $\phi\colon G\twoheadrightarrow H$ be an epimorphism of free groups with $H$
finitely generated. Then, the equalizer of any family of sections of $\phi$ is a free
factor of $H$.
\end{thm}

Here, a \emph{section} of $\phi\colon G\to H$ is a homomorphism going in the opposite
direction, $\sigma\colon H\to G$, and such that $\phi\sigma =Id\colon H\to H$.

Finally, the second evidence towards Conjecture~\ref{conj} was the following result,
proved some years later by Martino--Ventura in~\cite{MV}:

\begin{thm}[Martino--Ventura, \cite{MV}]\label{compression of fixed subgp in free gp}
Let $F$ be a finitely generated free group, and let $\B\subseteq \edo(F)$. Then, the
fixed subgroup $\fix \B$ is compressed in $F$.
\end{thm}

For more details about fixed subgroups in free groups, we refer the interested reader to
the survey~\cite{V} (which covers the history of this line of research up to 2002).

\bigskip

Let us briefly recall now what is known about fixed subgroups in surface groups. A
\emph{surface} group is the fundamental group, $G=\pi_1(X)$, of a connected closed
(possibly non-orientable) surface $X$. To fix the notation, we shall denote $\Sigma_g$
the closed orientable surface of genus $g\geqslant 0$, and
 $$
S_g= \pi_1(\Sigma_g )=\langle a_1,b_1,\ldots,a_g,b_g \mid [a_1, b_1]\cdots [a_g, b_g]
\rangle
 $$
its fundamental group (by convention, $S_0=\langle \, \mid \, \rangle$ stands for the
trivial group, the fundamental group of the sphere $\Sigma_0$); here, we use the
notation $[x,y]=xyx^{-1}y^{-1}$. And for the non-orientable case, we shall denote
$N\Sigma_k$ the connected sum of $k\geqslant 1$ projective planes, and
 $$
NS_k =\pi_1(N\Sigma_k)=\langle a_1,a_2,\ldots,a_k \mid a_1^2 \cdots a_k^2\rangle
 $$
its fundamental group. Note that, among surface groups, the only abelian ones are
$S_0=1$ (for the sphere), $S_1=\mathbb{Z}^2$ (for the torus), and
$NS_1=\mathbb{Z}/2\mathbb{Z}$ (for the projective plane).

It is well known that the Euler characteristic of orientable surfaces is
$\chi(\Sigma_g)=2-2g$, and of the non-orientable ones is $\chi(N\Sigma_k)=2-k$. Hence,
all surfaces have negative Euler characteristic except for the sphere $\Sigma_0$, the
torus $\Sigma_1$, the projective plain $N\Sigma_1$, and the Klein bottle $N\Sigma_2$
(homeomorphic to the connected sum of two projective plains). As can be seen below, many
results about automorphisms and endomorphisms will work in general for surfaces with
negative Euler characteristic; $S_0$, $S_1$, $NS_1$, and $NS_2$ will usually present
special and exceptional behaviour. We shall refer to the Euler characteristic also from
the groups, namely $\chi (S_g )=\chi (\Sigma_g )=2-2g$, and $\chi (NS_k )=\chi(
N\Sigma_k )=2-k$.

It is also well known that the standard sets of generators for surface groups given
above are minimal i.e., $\rk(S_g)=2g$ and $\rk(NS_k)=k$; this can be easily seen, for
example, by looking at their corresponding abelianizations. Furthermore, the well known
Freiheitssatz proved by Magnus, see~\cite[Theorem~5.1]{LS}, states that any proper
subset of this set of generators form a free basis is the subgroup they generate.

First results about fixed subgroups of surface groups are due to Jiang--Wang--Zhang, who
showed in~\cite{JWZ} that $\rk(\fix\phi)\leqslant \rk(G)$, for every endomorphism
$\phi\in \edo(G)$ of a surface group $G$ with $\chi(G)<0$. In the recent
paper~\cite{WZ}, Wu--Zhang extended it to the following results:

\begin{thm}[Wu--Zhang, \cite{WZ}]\label{WZ thm}
Let $G$ be a surface group with $\chi(G)<0$, and let $\B\subseteq \edo(G)$. Then,
\begin{itemize}
\item[(i)]$\rk (\fix\B)\leqslant \rk (G)$, with equality if and only if $\B=\{id\}$;
\item[(ii)] $\rk (\fix\B)\leqslant \frac{1}{2}\rk(G)$, if $\B$ contains a
    non-epimorphic endomorphism;
\item[(iii)] if $\B\subseteq \aut(G)$, then $\fix \B$ is inert in $G$.
\end{itemize}
\end{thm}

For equalizers of sections of homomorphisms from surface groups to free groups,
Wu--Zhang gave the following result, see~\cite[Proposition 4.7]{WZ}:

\begin{prop}[Wu--Zhang, \cite{WZ}]\label{equalizer of endomorphisms of surface group}
Let $G$ be a surface group with $\chi(G)<0$, and $F$ a finitely generated free group. If
$\phi: G\twoheadrightarrow F$ is an epimorphism, and $\B$ is a family of sections of
$\phi$, then
 $$
\rk(\eq\B )\leqslant \rk(F)\leqslant\frac{1}{2}\rk(G).
 $$
\end{prop}

In view of these results it seems reasonable to state the inertia Conjecture~\ref{conj} for
surface groups as well.

\begin{conj}\label{conj surfaces}
For an arbitrary family of endomorphisms $\B\subseteq \edo(F)$ of a surface group $G$,
$\fix\B$ is inert in $G$.
\end{conj}

The structure of the paper is as follows. We dedicate Section~\ref{free} to finitely
generated free groups: our main result in this context is Theorem~\ref{main free}, which
provides an alternative proof for Theorem~\ref{compression of fixed subgp in free gp},
and also gives a partial positive solution to Conjecture~\ref{conj} (see
Corollaries~\ref{cor free} and~\ref{cor free 2}).

We dedicate Section~\ref{surface} to surface groups: the advantage of our proof of
Theorem~\ref{main free} is that it easily translates into surface groups, see
Theorem~\ref{main surface} our main result in this context. As corollaries, we get some
new results giving supporting evidence to Conjecture~\ref{conj surfaces}: we prove
compression for fixed subgroups of arbitrary families of endomorphisms of a surface
group (see Corollary~\ref{compressed surface}), and we get partial positive solutions to
Conjecture~\ref{conj surfaces} (see Corollaries~\ref{cor surface} and~\ref{cor
surface2}).

Finally, in Section~\ref{products} we investigate the same issues (Bestvina--Handel
bound, compression, and inertia) both for automorphisms and endomorphism of groups of
the form $G=G_1\times \cdots \times G_n$, where $n\geqslant 0$ and each $G_i$ is either
a finitely generated free group or a surface group, $G_i =F_r,\, S_g,\, NS_k$ for some
$r\geqslant 1$, $g\geqslant 1$ or $k\geqslant 1$, respectively. In this context, we give
a characterization of those groups $G$ within this family for which $\rk (\fix \phi
)\leqslant \rk(G)$ for every $\phi \in \aut(G)$ (see Theorem~\ref{main products 1}), a
necessary condition for those satisfying that $\fix (\phi)$ is compressed in $G$ for
every $\phi \in \aut(G)$ (see Theorem~\ref{main products 2}), and a necessary condition
for those satisfying that $\fix (\phi)$ is inert in $G$ for every $\phi \in \aut(G)$
(see Theorem~\ref{main products 3}).

\section{Fixed points in free groups}\label{free}

In this context, we prove the following result for finitely generated free groups. The
proof (as well as the proof of later Theorem~\ref{main surface}) is an adaptation and
generalization of the proof of~\cite[Theorem~1.3]{WZ} to our situation. We hope this
contribution sheds light for the resolution of the full Conjecture~\ref{conj} in the
future.

\begin{thm}\label{main free}
Let $F$ be a finitely generated free group, let $\B\subseteq \edo(F)$ be an arbitrary
family of endomorphisms, let $\langle \B\rangle \leqslant \edo(F)$ be the submonoid
generated by $\B$, and let $\beta_0 \in \langle \B\rangle$ with image $\beta_0(F)$ of
minimal rank. Then, for every subgroup $K\leqslant F$ such that $\beta_0 (K)\cap \fix\B
\leqslant K$, we have $\rk (K\cap \fix\B)\leqslant \rk(K)$.
\end{thm}

\begin{proof}
Suppose $F$ is a finitely generated free group, $\B \subseteq \edo(F)$ is an arbitrary
family of endomorphisms of $F$, and $\langle \B\rangle$ is the closure of $\B$ in
$\edo(F)$ by composition (note also that $Id \in \langle \B\rangle$). Since, for any
$\alpha,\beta\in \B$, $\fix\alpha \,\cap\, \fix\beta\subseteq {\rm Fix}(\alpha\beta)$,
it is clear that $\fix \langle \B\rangle =\fix\B$ and so, the inequality we have to
prove does not change when replacing $\B$ to $\langle \B\rangle$. Hence we can assume,
without loss of generality, that $\B$ itself is a submonoid of $\edo(F)$ i.e., $\langle
\B\rangle =\B$.

Now choose $\beta_0\in \B$ such that
 $$
\rk(\beta_0(F))=\min\{\rk(\gamma(F)) \mid \gamma\in\B\}.
 $$
Thus all elements of $\B$ act injectively on $\beta_0 (F)$. Let $\beta_0\B
=\{\beta_0\gamma \mid \gamma\in \B \} \subseteq \B$. Since $\beta_0 \gamma(\beta_0
(F))\leqslant \beta_0 (F)$ we get, by restriction, a family $\beta_0\B|_{\beta_0(F)}$ of
injective endomorphisms of the finitely generated free group $\beta_0(F)$,
 $$
\beta_0 \gamma|_{\beta_0 (F)}\colon \beta_0 (F)\to \beta_0 (F).
 $$
By Theorem~\ref{inj fixed subgp of free gp inert}, $\fix(\beta_0 \B )=\fix(\beta_0 \B
|_{\beta_0(F)})$ is inert in $\beta_0(F)$ that is, for every $L\leqslant \beta_0(F)$, we
have
\begin{equation}\label{eq.1}
\rk(L\cap \fix(\beta_0\B))\leqslant\rk(L).
\end{equation}

Now let $K\leqslant F$ be a subgroup such that $\beta_0 (K)\cap \fix\B \leqslant K$; we
have to show that $\rk (K\cap \fix\B)\leqslant \rk(K)$. Let
 $$
E=\beta_0^{-1}(\beta_0(K)\cap \fix(\beta_0 \B))\leqslant F.
 $$
By construction, $\beta_0$ gives an epimorphism of free groups,
 $$
\beta_0|_E \colon E\twoheadrightarrow \beta_0(K)\cap \fix(\beta_0\B),
 $$
with image being finitely generated. On the other hand, every $\gamma\in \B$ restricts to a
section of $\beta_0|_E$, namely
 $$
\gamma|_{\beta_0(K)\cap \fix(\beta_0 \B)} \colon \beta_0(K)\cap \fix(\beta_0 \B)\to E;
$$
in fact, for every $x\in \beta_0(K)\cap \fix(\beta_0 \B)$, it is clear that
$\beta_0\gamma(x)=x$ and so, $\gamma(x)\in E$ and $\beta_0\gamma|_{\beta_0(K)\cap
\fix(\beta_0 \B)} =Id_{\beta_0(K)\cap \fix(\beta_0 \B)}$; in particular, taking
$\gamma=Id$, we have $\beta_0(K)\cap \fix(\beta_0 \B)\leqslant E$. Hence, by
Theorem~\ref{intersection of retracts of free gp} applied to this family of sections, we
obtain that $\eq(\B|_{\beta_0(K)\cap \fix(\beta_0\B)})$ is a free factor of
$\beta_0(K)\cap \fix(\beta_0\B)$. Since this family of sections contains the inclusion
of $\beta_0(K)\cap \fix(\beta_0\B)$ into $E$, we have
 $$
\begin{array}{rcl} \eq(\B|_{\beta_0(K)\cap \fix(\beta_0\B)}) & = & \fix(\B|_{\beta_0(K)\cap
\fix(\beta_0\B)}) \\ & = & \fix \B \cap \beta_0(K)\cap \fix(\beta_0\B) \\ & = & \beta_0(K)
\cap \fix \B \\ & = & K\cap \fix \B.
\end{array}
 $$
(For one of the inclusions in the last equality we use our assumption on $K$, the other
one is immediate.) Hence, using equation~(\ref{eq.1}) for $L=\beta_0(K)$, we conclude
 $$
\rk(K\cap \fix \B)\leqslant \rk(\beta_0(K)\cap \fix(\beta_0\B))\leqslant \rk(\beta_0(K))
\leqslant \rk(K),
 $$
completing the proof.
\end{proof}

As mentioned above, the argument in Theorem~\ref{main free} provides an alternative
proof for Theorem~\ref{compression of fixed subgp in free gp}, easier than the one given
in~\cite{MV}: every subgroup $K$ with $\fix \B \leqslant K\leqslant F$ clearly satisfies
the hypothesis $\beta_0 (K)\cap \fix\B \leqslant K$ and so, we have $\rk (\fix \B )=\rk
(K\cap \fix\B)\leqslant \rk(K)$. This shows compression of $\fix \B$.

Another interesting consequence of Theorem~\ref{main free} is the following partial
positive solution to Conjecture~\ref{conj}. We hope this helps to its full resolution in
the future.

\begin{cor}\label{cor free}
Let $F$ be a finitely generated free group, let $\B\subseteq \edo(F)$ be an arbitrary
family of endomorphisms, let $\langle \B\rangle \leqslant \edo(F)$ be the submonoid
generated by $\B$, and let $\beta_0 \in \langle \B\rangle$ with image $\beta_0(F)$ of
minimal rank. Then, $\fix \B$ is inert in $\beta_0(F)$. Moreover, if $\beta_0(F)$ is
inert in $F$ then $\fix \B$ is inert in $F$ as well.
\end{cor}

\begin{proof}
The first statement follows directly from Theorem~\ref{main free} as soon as we show
that every $K\leqslant \beta_0(F)$ satisfies the condition $\beta_0 (K)\cap \fix\B
\leqslant K$. Let $x\in \beta_0 (K)\cap \fix\B$; this implies that
$\beta_0(k)=x=\beta_0(x)$ for some $k\in K$. But both $k$ and $x$ belong to
$\beta_0(F)$, where $\beta_0$ is injective by the minimality condition in the definition
of $\beta_0$. Hence, $x=k\in K$.

For the second statement we only need to recall transitivity of the inertia property (if
$A\leqslant B\leqslant C$, and $A$ is inert in $B$, and $B$ is inert in $C$, then $A$ is
inert in $C$).
\end{proof}

\begin{cor}\label{cor free 2}
Let $F$ be a finitely generated free group, and let $\B\subseteq \edo(F)$ be an
arbitrary family of endomorphisms. If some composition of endomorphisms from $\B$ has
image of rank 1 or 2, then $\fix \B$ is inert in $F$.
\end{cor}

\begin{proof}
By assumption, the minimal rank of the image of the endomorphisms in $\langle \B\rangle$
is either 1 or 2. Let $\beta_0 \in \langle \B\rangle$ realize such minimum. Then
$\beta_0(F)$ is inert in $F$ (by the positive solution to H. Neumann conjecture or,
better, by the special case previously proved by Tardos~\cite{T}). Hence, by
Corollary~\ref{cor free}, $\fix \B$ is inert in $F$.
\end{proof}

\section{Fixed points in surface groups}\label{surface}

Before studying fixed subgroups in surface groups, let us remind the following folklore
facts about surface groups which will be used later. A first fact where the assumption of
negative Euler characteristic is crucial, is about the center and the centralizer of
non-trivial elements in surface groups. The following lemma can be easily deduced from the
fact that closed surfaces with negative Euler characteristic admit a hyperbolic metric (see
Theorem~1.2 and the subsequent comment in page~22 of~\cite{FM}).

\begin{lem}\label{centr}
Let $G$ be a surface group with $\chi (G)<0$. Then its center is trivial, $Z(G)=1$, and the
centralizer of any non-trivial element $1\neq g\in G$ is infinite cyclic, $\cen_G(g)\simeq
\mathbb{Z}$.
\end{lem}

\begin{rem}\label{centr exc}
The same result is true for free groups $F_r$ of rank $r\geqslant 2$; this will be
crucial for the arguments in Section~\ref{products}. However, it is no longer true for
the remaining groups within the family of finitely generated free and surface groups:
$F_0=S_0 =1$ is trivial, $F_1=\mathbb{Z}$, $S_1=\mathbb{Z}^2$ and $NS_1
=\mathbb{Z}/2\mathbb{Z}$ are abelian, and the \emph{Klein bottle group},
 $$
NS_2 = \langle a_1, a_2 \mid a_1^2a_2^2 =1\rangle \simeq \langle a,b \mid aba^{-1}b\rangle,
 $$
is not abelian but has center isomorphic to $\mathbb{Z}$, with generator $a_1^2=a^2$
(the isomorphism above is given by $a_1\mapsto a$, $a_2\mapsto a^{-1}b$); in fact,
$NS_2$ is virtually abelian since $\mathbb{Z}^2\simeq \langle a^2, b\rangle$ is an index
two normal subgroup of $NS_2$.
\end{rem}

The following are basic facts on surface groups, we include the proofs here for
completeness (see also~\cite[Lemma~2.7]{WZ}).

\begin{lem}\label{subgp of surface gp}
Let $G$ be a surface group with $\chi(G)<0$. Then,
\begin{itemize}
\item[(i)] if $H<G$ is a proper subgroup with $\rk(H)\leqslant \rk(G)$, then $H$ is a
    free group;
\item[(ii)] if $\phi\colon G\to G$ is a non-epimorphic endomorphism, then $\phi(G)$ is
    a free group of rank $\rk (\phi(G))\leqslant \frac{1}{2}\rk(G)$;
\item[(iii)] the group $G$ is both Hophian and co-Hopfian (i.e., all injective and all
    surjective endomorphisms of $G$ are, in fact, automorphisms).
\end{itemize}
\end{lem}

\begin{proof}
(i) follows easily from the fact that subgroups of $G$ are either free or finite index;
and the subgroups $H\leqslant G$ in this last family are again surface groups with
negative Euler characteristic and satisfying $\chi(H)/\chi(G)=[G:H]>1$ hence,
$\rk(H)>\rk(G)$.

For (ii), we recall the fact that the \emph{inner rank} of $G$ i.e., the maximal rank of
a free quotient of $G$, is at most $\frac{1}{2}\rk(G)$ (see Lyndon--Schupp~\cite[page
52]{LS}).

Finally, by a classical result, surface groups are residually finite (see~\cite{Hem} for
a short proof) and so, Hophian. On the other hand, if $\phi\in \edo(G)$ is injective
then the image $\phi(G)$ is isomorphic to $G$, but $\rk(\phi(G))\leqslant \rk(G)$; by
(i), $\phi \in \aut(G)$. This completes the proof of (iii).
\end{proof}

\begin{rem}
Again, the situation is different without the hypothesis of negative Euler characteristic.
For the case of the torus, $S_1=\mathbb{Z}^2$ violates (i), (ii) and the co-Hopfianity. For
the projective plane, $NS_1 =\mathbb{Z}/2\mathbb{Z}$ satisfies the Lemma but with trivial
meaning. And, finally, the Klein bottle group $NS_2$ violates again (i), (ii) and the
co-hopfianity, since $a\mapsto a^3$, $b\mapsto b$ defines an injective non-surjective
endomorphism whose image is isomorphic to $NS_2$ itself. (Both $\mathbb{Z}^2$ and $NS_2$ are
Hopfian, like all surface groups without exception.)
\end{rem}

The proof of Theorem~\ref{main free} works for surface groups as well, with some extra
arguments distinguishing whether certain involved subgroups are free or finite index. We
reproduce that argument here to highlight these important points. It works for surface
groups with negative Euler characteristic; however, for the exceptional ones one can
directly prove the inertia conjecture.

\begin{prop}\label{inertia small}
Let $G$ be either $F_0=S_0 =1$, or $S_1=\mathbb{Z}^2$, or $NS_1 =\mathbb{Z}/2\mathbb{Z}$, or
$NS_2$, and let $\B\subseteq\edo(G)$ be an arbitrary family of endomorphisms. Then, $\fix\B$
is inert in $G$.
\end{prop}

\begin{proof}
For $F_0=S_0 =1$ and $NS_1 =\mathbb{Z}/2\mathbb{Z}$ the result is trivially true. In
$S_1=\mathbb{Z}^2$ subgroups satisfy, in general, the implication $H\leqslant K \,
\Rightarrow \, \rk(H)\leqslant \rk (K)$ so, again, the result is clearly true.

For the Klein bottle group $NS_2$, Wu--Zhang showed in \cite[Example 6.2]{WZ} that, for
every $Id\neq \phi\in\edo(NS_2)$, we always have $\fix \phi \simeq \Z^2, \Z, 1$. Thus,
for every $\{ Id\}\neq \B \subseteq \edo(NS_2)$, we also have $\fix\B \simeq \Z^2, \Z,
1$ and so, $\fix\B$ is inert in $NS_2$.
\end{proof}

\begin{thm}\label{main surface}
Let $G$ be a surface group, let $\B\subseteq \edo(G)$ be an arbitrary family of
endomorphisms, let $\langle \B\rangle \leqslant \edo(G)$ be the submonoid generated by $\B$,
and let $\beta_0 \in \langle \B\rangle$ with image $\beta_0(G)$ of minimal rank. Then, for
every subgroup $K\leqslant G$ such that $\beta_0 (K)\cap \fix\B \leqslant K$, we have $\rk
(K\cap \fix\B)\leqslant \rk(K)$.
\end{thm}

\begin{proof}
If $\chi(G)\geqslant 0$ then Proposition~\ref{inertia small} gives us inertia of $\fix \B$
and so, the inequality $\rk (K\cap \fix\B)\leqslant \rk(K)$ holds for every $K\leqslant G$
without conditions. Let us assume then $\chi(G)<0$.

As in the proof of Theorem~\ref{main free}, we may assume that $\B$ contains the identity
and it is closed under composition. If $\B$ consists of epimorphisms, then $\B\subseteq
\aut(G)$, Theorem~\ref{WZ thm}(iii) tells us that $\fix \B$ is inert in $G$, and we are
done.

So, let us assume that $\B$ contains at least one non-epimorphic endomorphism; by
Lemma~\ref{subgp of surface gp}(ii), its image will be a free group. As above, choose
$\beta_0\in \B$ such that $\beta_0(G)$ is a free group of minimal rank. All elements of
$\B$ must then act injectively on $\beta_0(G)$. As above, we consider the subset
$\beta_0\B =\{\beta_0\gamma \mid \gamma\in \B\} \subseteq \B$, their restrictions to
$\beta_0(G)$ give a family $\beta_0\B|_{\beta_0(G)}$ of injective endomorphisms of the
free group $\beta_0(G)$, namely $\beta_0\gamma|_{\beta_0(G)} \colon \beta_0(G)\to
\beta_0(G)$, and, by Theorem~\ref{inj fixed subgp of free gp inert}, $\fix(\beta_0 \B
)=\fix(\beta_0 \B |_{\beta_0(G)})$ is inert in $\beta_0(G)$ i.e., for every $L\leqslant
\beta_0(G)$, we have
\begin{equation}\label{eq.2}
\rk(L\cap \fix(\beta_0\B))\leqslant\rk(L).
\end{equation}

Now let $K\leqslant G$ be a subgroup satisfying $\beta_0 (K)\cap \fix\B \leqslant K$; we
have to show that $\rk (K\cap \fix\B)\leqslant \rk(K)$. As above, we consider
$E=\beta_0^{-1}(\beta_0(K)\cap \fix(\beta_0 \B))\leqslant G$, $\beta_0$ restricts to an
epimorphism $\beta_0|_E \colon E\twoheadrightarrow \beta_0(K)\cap \fix(\beta_0\B)$ whose
image is free and finitely generated, and every $\gamma\in \B$ restricts to a section of
$\beta_0|_E$, namely
 $$
\gamma|_{\beta_0(K)\cap \fix(\beta_0 \B)} \colon \beta_0(K)\cap \fix(\beta_0 \B)\to E.
 $$
Now, let us distinguish whether $E\leqslant G$ is a free group (like in the proof of
Theorem~\ref{main free}) or a surface group. In the first case, Theorem~\ref{intersection of
retracts of free gp} tells us that $\eq(\B|_{\beta_0(K)\cap \fix(\beta_0\B)})$ is a free
factor of $\beta_0(K)\cap \fix(\beta_0\B)$ and so,
 $$
\rk(\eq(\B|_{\beta_0(K)\cap \fix(\beta_0\B)}))\leqslant \rk(\beta_0(K)\cap \fix(
\beta_0\B)).
 $$
Otherwise, Proposition~\ref{equalizer of endomorphisms of surface group} gives us this
inequality directly.

Finally, the exact same argument as above shows that $\eq(\B|_{\beta_0(K)\cap
\fix(\beta_0\B)}) = K\cap \fix \B$, using here our assumption on $K$. Hence, by
equation~(\ref{eq.2}) applied to $L=\beta_0(K)$, we conclude
 $$
\rk(K\cap \fix \B )\leqslant \rk(\beta_0(K)\cap \fix(\beta_0\B))\leqslant \rk(\beta_0(K))
\leqslant \rk(K),
 $$
completing the proof.
\end{proof}

As a first corollary, we get compression of fixed subgroups of arbitrary families of
endomorphisms, the analogous result to Theorem~\ref{compression of fixed subgp in free gp}
for surface groups.

\begin{cor}\label{compressed surface}
Let $G$ be a surface group, and let $\B\subseteq \edo(G)$. Then, $\fix \B$ is compressed in
$G$.
\end{cor}

\begin{proof}
This is a direct consequence of Theorem~\ref{main surface}, after observing that $\fix\B
\leqslant K\leqslant G$ trivially implies $\beta_0 (K)\cap \fix\B \leqslant K$.
\end{proof}

Finally, we can also get the corresponding partial positive solution to
Conjecture~\ref{conj surfaces}. We hope this helps to its full resolution in the future.

\begin{cor}\label{cor surface}
Let $G$ be a surface group, let $\B\subseteq \edo(G)$ be an arbitrary family of
endomorphisms, let $\langle \B\rangle \leqslant \edo(G)$ be the submonoid generated by
$\B$, and let $\beta_0 \in \langle \B\rangle$ with image $\beta_0(G)$ of minimal rank.
Then, $\fix \B$ is inert in $\beta_0(G)$. Moreover, if $\beta_0(G)$ is inert in $G$ then
$\fix \B$ is inert in $G$ as well.
\end{cor}

\begin{proof}
The exact same argument as in Corollary~\ref{cor free} works here.
\end{proof}

\begin{cor}\label{cor surface2}
Let $\B\subseteq\edo(NS_3)$ be an arbitrary family of endomorphisms of the surface group
$NS_3$. Then, $\fix\B$ is inert in $NS_3$.
\end{cor}

\begin{proof}
If $\B\subseteq\aut(NS_3)$, then $\fix\B$ is inert in $NS_3$ according to Theorem \ref{WZ
thm}(iii). Otherwise, Lemma~\ref{subgp of surface gp}(iii) and (ii) imply that some
$\beta_0\in \B$ has image being free of minimal rank, and $\rk(\beta_0(NS_3))\leqslant
\lfloor \frac{1}{2}\rk(NS_3)\rfloor =1$. Then, $\fix \B \leqslant \beta_0(NS_3)$ is also
cyclic and so, inert in $NS_3$.
\end{proof}

\begin{rem}
I. Mineyev claimed (but did not explicitly prove) that Hanna Neumann conjecture also
holds for surface groups (see~\cite[Section 8]{M1}). If this were correct, then every
rank two subgroup $A$ of a surface group $G$ would be inert in $G$, and the argument
above would also work for ambient surface groups of rank up to five: any family of
endomorphisms $\B\subseteq \edo(G)$ of a surface group $G$ with $\rk(G)\leqslant 5$
would satisfy that $\fix\B$ is inert in $G$ (i.e., Corollary~\ref{cor surface2} would
also be valid for $S_2$, $NS_4$ and $NS_5$). In fact, the case $\B\subseteq \aut(G)$ is
covered by Theorem~\ref{WZ thm}(iii); and otherwise, there is some $\beta_0\in \B$ with
free image of rank $\rk(\beta_0(G))\leqslant \lfloor \frac{1}{2}\rk(G)\rfloor =2$ and
so, by Corollary~\ref{cor surface}, $\fix\B$ is inert in $\beta_0(G)$, which is inert in
$G$.
\end{rem}

\section{Fixed points in direct products}\label{products}

In the previous sections we considered, separately, finitely generated free groups, and
surface groups. Now, we shall study direct products of finitely many such groups (with
possible repetitions) i.e., groups of the form $G=G_1\times \cdots \times G_n$, where
$n\geqslant 1$ and each $G_i$ is either a finitely generated free group $F_r$, $r\geqslant
1$, or an orientable surface group $S_g$, $g\geqslant 1$, or a non-orientable surface group
$NS_k$, $k\geqslant 1$. For short (and only within the scope of the present preprint), we
shall call such $G$ a \emph{product group}. We shall consider automorphisms and
endomorphisms of product groups and shall analyze the properties of their fixed subgroups.

We need to begin with some basic algebraic properties of product groups. First note that,
for arbitrary groups $A$ and $B$, $\rk(A\times B)\leqslant \rk(A)+\rk(B)$. Sometimes this
inequality is strict (as illustrated with the well-known fact that a direct product of two
finite cyclic groups of coprime orders is again cyclic), but in the case of product groups
it is always an equality.

\begin{lem}\label{ranks}
Let $G=G_1\times \cdots \times G_n$, where each $G_i$ is either a finitely generated
free group or a surface group. Then, $\rk(G)=\rk(G_1)+\cdots +\rk(G_n)$.
\end{lem}

\begin{proof}
By the form of the $G_i$'s, the torsion part of $G_i^{\rm \, ab}$ is either trivial or
$\mathbb{Z}/2\mathbb{Z}$; furthermore, $\rk(G_i^{\rm \, ab})=\rk(G_i)$. Then,
 $$
\begin{array}{rcl}
\rk(G) \geqslant \rk(G^{\rm \, ab}) & = & \rk(G_1^{\rm \, ab}\times \cdots \times
G_n^{\rm \, ab}) \\ & = & \rk(G_1^{\rm \, ab})+\cdots +\rk(G_n^{\rm \, ab}) \\ & = & \rk(G_1)+\cdots
+\rk(G_n) \\ & \geqslant & \rk(G_1\times \cdots \times G_n) \\ & = & \rk(G),
\end{array}
 $$
and the two inequalities are equalities. Hence, $\rk(G)=\rk(G_1)+\cdots +\rk(G_n)$.
\end{proof}

Note that the center of a direct product of groups, $A\times B$, is the product of the
corresponding centers, $Z(A\times B)=Z(A)\times Z(B)$. Also, the centralizer of an element
$(a,b)\in A\times B$ is clearly the direct product of centralizers of its respectives
components, $\cen_{A\times B}(a,b)=\cen_A(a)\times \cen_B(b)$; in particular, $\cen_{A\times
B}(a,1)=\cen_A(a)\times B$ and $\cen_{A\times B}(1,b)=A\times \cen_B(b)$.

Note also that, among our building blocks (namely, $F_r$ for $r\geqslant 1$, $S_g$ for
$g\geqslant 1$, and $NS_k$ for $k\geqslant 1$) the only abelian ones are $F_1
=\mathbb{Z}$, $S_1=\mathbb{Z}^2$ and $NS_1=\mathbb{Z}/2\mathbb{Z}$. Also, they all have
trivial center and infinite cyclic centralizers for non-trivial elements, except for
$F_1$, $S_1$, $NS_1$ and $NS_2$ (see Lemma~\ref{centr} and Remark~\ref{centr exc}).
Using this, we can easily deduce how is the center and the centralizer of an arbitrary
element in a product group. The following lemma will be crucial in the forthcoming
arguments.

\begin{lem}\label{center-centr}
Let $G=G_1\times \cdots \times G_n$, $n\geqslant 1$, be a product group, where each $G_i$ is
a free or a surface group. Then,
 $$
\begin{array}{rcl}
Z(G)=1 & \Longleftrightarrow & \text{each } G_i \text{ is free non-abelian, or a surface
group with } \chi(G_i)<0 \\ & \Longleftrightarrow & \forall i=1,\ldots ,n,\,\, G_i\simeq
F_r \text{ with } r\geqslant 2, \text{ or } G_i\simeq S_g \text{ with } g\geqslant 2,
\text{ or } \\ & & G_i\simeq NS_k \text{ with } k\geqslant 3. \end{array}
 $$
Furthermore, in the case $Z(G)=1$, every element $(g_1,\ldots ,g_n)\in G$ satisfies
$\cen_G(g_1,\ldots ,g_n)\simeq \widehat{G_1}\times \cdots \times \widehat{G_n}$, where
$\widehat{G_i}$ is $G_i$ if $g_i=1$, or $\mathbb{Z}$ if $g_i\neq 1$.
\end{lem}

Lemma~\ref{center-centr} separates product groups $G=G_1\times \cdots \times G_n$ into
three different types, namely: (i) those for which $Z(G_i)\neq 1$ for all $i=1,\ldots
,n$, they will be called of \emph{euclidean} type, and are precisely the groups of the
form $G=NS_2^{\, \ell} \times \mathbb{Z}^p \times (\mathbb{Z}/2\mathbb{Z})^q$, for some
integers $\ell, p,q\geqslant 0$; (ii) those for which $Z(G_i)=1$ for all $i=1,\ldots
,n$, they will be called of \emph{hyperbolic} type and are the finite direct products of
$F_r$'s with $r\geqslant 2$, $S_g$'s with $g\geqslant 2$, and $NS_k$'s with $k\geqslant
3$; and (iii) those mixing the two behaviours. With this language,
Lemma~\ref{center-centr} says that $Z(G)=1$ if and only if $G$ is of hyperbolic type.
Euclidean and hyperbolic product groups will play an important role in the subsequent
arguments.

For product groups the global group determines the number of components and the
components themselves (except for the case of $\mathbb{Z}^2$ being isomorphic to
$\mathbb{Z}\times \mathbb{Z}$); see the next proposition for details. This contrasts
with the general situation where $\Z\times A \simeq \Z\times B$ does not imply $A\simeq
B$, see~\cite{Hir}.

\begin{prop}\label{isom}
Let $G=G_1\times \cdots \times G_n$ and $H=H_1\times \cdots \times H_m$, $n,m\geqslant
1$, be two product groups of hyperbolic type (where each $G_i$ and $H_j$ is a
non-abelian free group or a surface group with negative Euler characteristic). Then,
$G\simeq H$ if and only if $n=m$ and $G_i\simeq H_i$ up to reordering.
\end{prop}

\begin{proof}
The ``if\," part is obvious.

For the convers, suppose $G\simeq H$ and, by symmetry, $n\leqslant m$. By
Lemma~\ref{center-centr}, $Z(G)=Z(H)=1$. Let us prove the implication by induction on
$n$.

For $n=1$, and looking at centralizers, we immediately deduce $m=1$ and $G_1 =G\simeq
H=H_1$.

So suppose $n\geqslant 2$, assume the implication is true for $n-1$, and let $\phi
\colon G\to H$ be an isomorphism. Let $1\neq g\in G_n$ be one of the standard generators
for $G_n$, and consider the element $(1,\ldots ,1,g)\in G$. By Lemma~\ref{center-centr},
$\cen_G(1,\ldots ,1,g)\simeq G_1\times \cdots \times G_{n-1}\times \mathbb{Z}$. And,
again by Lemma~\ref{center-centr}, $\cen_H(\phi(1,\ldots ,1,g))\simeq H_{j_1}\times
\cdots \times H_{j_{m-k}}\times \mathbb{Z}^k$, where $\{ j_1, \ldots , j_{m-k}\}$ are
the positions of the trivial coordinates in $\phi(1,\ldots ,1,g)$, and $k$ is its number
of non-trivial coordinates. Then, $G_1\times \cdots \times G_{n-1}\times
\mathbb{Z}\simeq \cen_G(1,\ldots ,1,g)\simeq \cen_H(\phi(1,\ldots ,1,g))\simeq
H_{j_1}\times \cdots \times H_{j_{m-k}}\times \mathbb{Z}^k$. From here, taking centers,
we get $k=1$; this means that $\phi(1,\ldots,1,g)$ has only one non-trivial coordinate,
say at position $j$. In principle $j$ depends on $g$, but a straightforward argument
shows it does not (any pair of different generators $g,g'$ from the standard
presentation for $G_n$ do not commute to themselves and so, $\phi(1,\ldots,1,g)$ and
$\phi(1,\ldots,1,g')$ do not commute either; this implies they both have their unique
non-trivial coordinate at the same position $j$). Up to reordering the coordinates of
$H$ if necessary, this means that $\phi$ restricts to an injective morphism $\phi|_{G_n}
\colon 1\times \cdots \times 1\times G_n \to 1\times \cdots \times 1\times H_m$. And
global surjectivity of $\phi$ together with an argument as above applied to $\phi^{-1}$
tells us that $\phi|_{G_n}$ is surjective. Hence, $G_n\simeq H_m$.

Finally, factoring out the centers in the above isomorphism between centralizers, we get
$G_1\times \cdots \times G_{n-1} \simeq H_1\times \cdots \times H_{m-1}$. By the
inductive hypothesis, $n-1=m-1$ (so, $n=m$) and, again up to reordering, $G_i\simeq H_i$
for all $i=1,\ldots ,n-1$.
\end{proof}

With similar arguments we can describe \emph{all} automorphisms of a product group $G$
with trivial center: up to permuting the coordinates corresponding to isomorphic
$G_i$'s, all automorphisms of $G$ will be rectangular i.e., a product of individual
automorphisms of the coordinates. After introducing the necessary notation, we establish
this fact formally in the following proposition.

For a product group $G=G_1\times \cdots \times G_n$, we can collect together the coordinates
corresponding to isomorphic $G_i$'s and present it in the form $G=G_1^{\,n_1}\times \cdots
\times G_m^{\,n_m}$, where $n_i\geqslant 1$ and $G_i\not\simeq G_j$ for different
$i,j=1,\ldots ,m$; of course, $n=n_1+\cdots +n_m$. When we need to distinguish between the
coordinates of $G_i^{\,n_i}$, we shall use the notation $G_i^{\,n_i}=G_{i,1}\times \cdots
\times G_{i,n_i}$, where $G_{i,j}=G_i$ for all $j=1,\ldots ,n_i$. Assuming a small risk of
confusion, we shall use both notations simultaneously (the meaning being clear from the
context at any time); we shall refer to them as the \emph{global} notation and the
\emph{block} notation, respectively.

Given an automorphism of each coordinate, $\phi_i\in \aut(G_i)$ for $i=1,\ldots ,n$,
their product $\phi =\prod_{i=1}^n \phi_i =\phi_1\times \cdots \times \phi_n \colon G\to
G$, $(g_1,\ldots ,g_n)\mapsto (\phi_1(g_1), \ldots ,\phi_n(g_n))$, is clearly an
automorphism of $G$; let us refer to such automorphisms of $G$ as the \emph{rectangular}
ones. On the other hand, given a permutation $\sigma \in S_{n_i}$ of the set of indices
$\{ 1,\ldots ,n_i \}$, the automorphism of $G_i^{\, n_i}$ defined by $(g_1,\ldots
,g_{n_i})\mapsto (g_{\sigma(1)},\ldots ,g_{\sigma(n_i)})$ extends to an automorphism of
$G$ by fixing the rest of coordinates; abusing notation, we shall denote both of them by
$\sigma$ (so, $\sigma \in S_{n_i}$, $\sigma \in \aut (G_i^{\, n_i})$, and $\sigma \in
\aut(G)$). Note that if $\sigma\in S_{n_i}$ and $\tau\in S_{n_j}$ with $i\neq j$, then
the corresponding automorphisms $\sigma,\, \tau \in \aut(G)$ act on supports with
trivial intersection and so they commute, $\sigma\tau=\tau\sigma$.

It is easy to see that some product groups admit non-rectangular automorphisms, even up
to permutation. For instance, take $G=F_2 \times \mathbb{Z}=\langle a,b\rangle \times
\langle c\rangle $ and $\phi \colon G\to G$, $a\mapsto ca$, $b\mapsto b$, $c\mapsto c$,
is such an example. However, note that for such a construction to be a well defined
automorphism, it is essential that $c$ is a central element in $G$ (in $F_2 \times F_2
=\langle a,b\rangle \times \langle c,d\rangle$, trying to send $a\mapsto ca$, $b\mapsto
b$, $c\mapsto c$, $d\mapsto d$ does not work because $c$ does not commute with $d$). The
following proposition states that the presence of non-trivial central elements in $G$
are necessary to have non-rectangular automorphisms i.e., if $G$ is of hyperbolic type
then every automorphism of $G$ is rectangular up to permutation.

\begin{prop}\label{autos}
Let $G=G_1^{\,n_1}\times \cdots \times G_m^{\,n_m}$ be a product group, where
$m\geqslant 1$, $n_i\geqslant 1$, $G_i\not\simeq G_j$ for $i\neq j$, and each $G_i$ is a
free group or a surface group. If $G$ is of hyperbolic type then, for every $\phi \in
\aut(G)$, there exist automorphisms $\phi_{i,j}\in \aut(G_i)$ and permutations $\sigma_i
\in S_{n_i}$, such that
 $$
\phi =\sigma_1 \circ\cdots \circ \sigma_m \circ (\prod_{i=1}^m \prod_{j=1}^{n_i}
\phi_{i,j}) = \prod_{i=1}^m (\sigma_i \circ 
\prod_{j=1}^{n_i} \phi_{i,j}).
 $$
\end{prop}

\begin{proof}
Let $\phi \in \aut(G)$. The exact same argument as in the inductive step of the proof
of Proposition~\ref{isom} applied to each coordinate shows that, for every $i=1,\ldots
,n$, there exists $j=1,\ldots ,n$ such that $\phi$ maps bijectively elements of the
subgroup $G_i\leqslant G$ to elements of the subgroup $G_j\leqslant G$ (here we are
using the global notation). Furthermore, it is clear that $i\mapsto j$ defines a
permutation $\sigma$ of $\{ 1,\ldots ,n\}$. In other words, there exists a permutation
$\sigma\in S_n$ and automorphisms $\phi_i \in \aut(G_i)$ such that $\phi =\sigma \circ
(\prod_{i=1}^n \phi_i)$.

Finally, $\phi_i$ is an isomorphism from $G_i$ to $G_{\sigma(i)}$ so, $\sigma$ must preserve
the isomorphism blocks, say $\sigma=\sigma_1\circ \cdots \circ \sigma_m$ for some $\sigma_i
\in S_{n_i}$, $i=1,\ldots ,m$. This concludes the proof (the equality in the statement is
expressed in the block notation).
\end{proof}

Now we develop a technical result, and an interesting construction providing
automorphisms whose fixed subgroups have rank bigger than the ambient rank (so,
violating Bestvina--Handel inequality).

\begin{lem}
Let $A_i$, $i=1,\ldots ,n$, be a family of groups such that every $H\leqslant A_i$
satisfies $\rk(H)\leqslant \rk(A_i)$. Then, $\rk(H)\leqslant \rk(A_1)+\cdots +\rk(A_n)$
holds for every $H\leqslant A_1\times \cdots \times A_n$.
\end{lem}

\begin{proof}
Let us do induction on $n$, the case $n=1$ being trivial.

Suppose the result is true for $n-1$, and let $H\leqslant A_1\times \cdots \times A_n$.
By induction, the image of $H$ under the canonical projection $\pi \colon A_1\times
\cdots \times A_n \twoheadrightarrow A_1\times \cdots \times A_{n-1}$ admits a set of at
most $\rk(\pi(H))\leqslant \rk(A_1)+\cdots +\rk(A_{n-1})$ generators. Choosing a
preimage in $H$ of each one, and adding to this set a set of generators for $H\cap
\ker\pi =H\cap A_n \leqslant A_n$, we obtain a set of generators for $H$. Hence,
$\rk(H)\leqslant \rk(\pi(H))+\rk(H\cap A_n)\leqslant \rk(A_1)+\cdots
+\rk(A_{n-1})+\rk(A_n)$, as we wanted to prove.
\end{proof}

\begin{cor}\label{facil}
Let $G=NS_2^{\, \ell}\times \mathbb{Z}^p \times (\mathbb{Z}/2\mathbb{Z})^q $ for some
integers $\ell, p, q\geqslant 0$. Then any subgroup $H\leqslant G$ satisfies
$\rk(H)\leqslant \rk(G)=2\ell +p+q$.
\end{cor}

\begin{proof}
This is a direct consequence of the previous lemma, after showing that every subgroup of
$NS_2$ has rank at most 2. And this is true by the following argument: using the
notation from Remark~\ref{centr exc}, it is easy to see that the abelianization short
exact sequence for $NS_2$ is
 $$
1\longrightarrow [NS_2, NS_2]=\langle b^2\rangle \longrightarrow NS_2 \longrightarrow
NS_2^{\rm \, \, ab}\longrightarrow 1,
 $$
where $NS_2^{\rm \, ab}=\langle \overline{a},\overline{b} \mid
[\overline{a},\overline{b}],\, \overline{b}^2\rangle \simeq \Z \times \Z/2\Z$. Given
$H\leqslant NS_2$, $\pi(H)\leqslant \Z \times \Z/2\Z$ and so, $\rk(\pi(H))\leqslant 2$.
And $H\cap [NS_2 : NS_2]=\langle b^{2\eta }\rangle\leqslant \langle b^2\rangle$ for some
$\eta\geqslant 0$ and so, $\rk(H\cap [NS_2 : NS_2])\leqslant 1$. Hence, $\rk(H)\leqslant
2+1=3$. But in the case $\rk(\pi(H))=2$, it must be $\pi(H)=\langle \overline{a}^{\nu},
\overline{b}\rangle$ for some $\nu \geqslant 1$. Hence, choosing preimages in $H$, say
$a^{\nu} b^{2\alpha}$ and $b^{1+2\beta}$, of $\overline{a}^{\nu}$ and $\overline{b}$,
respectively, we have $H=\langle a^{\nu}b^{2\alpha}, b^{1+2\beta}, b^{2\eta}\rangle
=\langle a^{\nu}b^{2\alpha}, b^{\tau}\rangle$, where $\tau =\operatorname{gcd}(1+2\beta,
2\eta)$. Therefore, in any case $\rk(H)\leqslant 2$, as we wanted to see.
\end{proof}

\begin{prop}\label{exs no BH}
Let $G=G_1\times \cdots \times G_n$, $n\geqslant 1$, be a product group with $G_1$
containing a non-trivial central element $1\neq t\in Z(G_1)$, and with $Z(G_2)=1$. Then,
there exists an automorphism $\phi\in \aut(G)$ such that $\rk(\fix \phi)>\rk (G)$.
\end{prop}

\begin{proof}
By Lemma~\ref{center-centr} and the facts $Z(F_1)=F_1$, $Z(S_1)=S_1$, $Z(NS_1)=NS_1$ and
$Z(NS_2)\simeq \mathbb{Z}$, we deduce that the order of $t\neq 1$ is either two or
infinite, $o(t)=2, \infty$. On the other hand, $G_2$ is either $F_r$ with $r\geqslant
2$, or $S_g$ with $g\geqslant 2$, or $NS_k$ with $k\geqslant 3$.

Suppose $G_2 =F_r =\langle a_1, \ldots ,a_r \mid \, \rangle$ with $r\geqslant 2$. Map
$a_1$ to $ta_1$, and fix all the other standard generators of $G$. This determines a
well defined automorphism $\phi \in \aut(G)$ sending $w(a_1,\ldots ,a_r)$ to
$w(ta_1,a_2, \ldots ,a_r)=t^{|w|_1}w(a_1,\ldots ,a_r)$, where $|w|_1 \in \mathbb{Z}$ is
the total $a_1$-exponent of $w\in G_2$. Hence,
 $$
\fix\phi =G_1 \times \{ w\in G_2 \mid |w|_1\equiv 0\} \times G_3\times \cdots \times G_n,
 $$
where $\equiv$ means equality of integers modulo $o(t)$ (by convention, read
$\mathbb{Z}/\infty\mathbb{Z}$ as just $\mathbb{Z}$). Now consider the projection
$\pi\colon G_2\twoheadrightarrow \mathbb{Z}/o(t)\mathbb{Z}$, $w\mapsto |w|_1$. Its
kernel, $\ker \pi$, is a normal subgroup of $G_2 =F_r$ of either infinite index (and so,
infinitely generated) or of index 2 (and so, $\rk(\ker \pi)=1+2(r-1)=2r-1$, according to
the Schreier index formula for free groups). In both cases, $\rk(\ker \pi)>r=\rk(G_2)$
and, by Lemma~\ref{ranks}, $\rk(\fix \phi)=\rk(G_1)+\rk(\ker \pi)+\rk(G_3)+\cdots
+\rk(G_n)>\rk(G)$. (In the infinite order case, this example was also considered for
similar reasons in~\cite{DeV, Z}.)

Now suppose $G_2=S_g =\langle a_1,b_1,\ldots,a_g,b_g \mid [a_1, b_1]\cdots [a_g,
b_g]\rangle$ with $g\geqslant 2$. Consider the automorphism $\phi \in \aut(G)$ defined
in the same way, namely mapping $a_1$ to $ta_1$ and fixing all the other standard
generators of $G$ (this determines a well defined automorphism of $G$ because $t$
commutes with $b_1$). As above, $w(a_1,b_1,\ldots ,a_g,b_g)$ maps to $w(ta_1,b_1,\ldots
,a_g, b_g )=t^{|w|_1}w(a_1,b_1,\ldots ,a_g,b_g)$, where $|w|_1 \in \mathbb{Z}$ is the
total $a_1$-exponent of $w\in G_2$ (which still makes sense because the defining
relation in $G_2$ has total $a_1$-exponent equal to zero). Hence, as above,
 $$
\fix\phi =G_1 \times \{ w\in G_2 \mid |w|_1\equiv 0\} \times G_3\times \cdots \times G_n,
 $$
where $\equiv$ means equality of integers modulo $o(t)$. The argument proceeds and
concludes like above, after proving that the rank of the kernel of $\pi\colon
G_2\twoheadrightarrow \mathbb{Z}/o(t)\mathbb{Z}$, $w\mapsto |w|_1$ is, again, strictly
bigger than $\rk(G_2)=2g$ (note that $\ker \pi$ is either a free group or a surface
group again so, Lemma~\ref{ranks} still applies). If $o(t)=2$, this is true because
$\ker\pi$ is a subgroup of index two in $G_2$, and so a surface group of bigger rank.
And if $o(t)=\infty$ then $\ker\pi$ is infinitely generated by the following argument:
$\ker \pi$ is a subgroup of infinite index in $G_2$ (and so free), but maximal as a free
subgroup: in fact, for every $x\in G_2\setminus \ker \pi$, we have $[G_2 : \langle \ker
\pi, x\rangle]=[\mathbb{Z} : \langle \pi(x)\rangle]=|\pi(x)|<\infty$ and so, $\langle
\ker \pi, x\rangle$ is a surface group of Euler characteristic equal to $[G_2 : \langle
\ker \pi, x\rangle]\chi(G_2)=|\pi(x)|(2-2g)$ and thus, of rank $2+|\pi(x)|(2g-2)$.
Choosing $x$ appropriately, this rank is arbitrarily big and therefore $\ker \pi$ cannot
be finitely generated.

Finally, suppose $G_2=\langle a_1,a_2,\ldots,a_k \mid a_1^2 \cdots a_k^2\rangle$ with
$k\geqslant 3$. Map $a_1$ to $ta_1$, $a_2$ to $t^{-1}a_2$ and fix all the other standard
generators of $G$ (this determines a well defined $\phi \in \aut(G)$ because $t$
commutes with both $a_1$ and $a_2$). Observe that now, because of the form of the
defining relation of $G_2$, the ``total $a_i$-exponent" of an element of $w\in G$ is not
well defined. However, the difference of two of them, say $|w|_1-|w|_2\in \mathbb{Z}$,
it really is; in other words, $G_2\twoheadrightarrow \mathbb{Z}$, $w\mapsto
|w|_1-|w|_2$, is a well defined morphism from $G_2$ onto $\mathbb{Z}$. Composing it with
the canonical projection, we get an epimorphism $\pi \colon G_2 \twoheadrightarrow
\mathbb{Z}/o(t)\mathbb{Z}$ and, since $\phi$ maps $w(a_1,\ldots ,a_k)$ to
$w(ta_1,t^{-1}a_2, a_3, \ldots ,a_k)=t^{|w|_1-|w|_2}w(a_1,\ldots ,a_k)$, we deduce like
in the above cases that
 $$
\fix\phi =G_1 \times \ker \pi \times G_3\times \cdots \times G_n.
 $$
The same argument as above shows that $\rk(\ker \pi)>\rk(G_2)$ and so,
$\rk(\fix\phi)>\rk(G)$, completing the proof.
\end{proof}

The construction in Proposition~\ref{exs no BH} is, essentially, the only way to produce
automorphisms whose fixed subgroups have rank bigger than the ambient group. This is the
contents of the following result, characterizing exactly which product groups satisfy
Bestvina--Handel inequality.

\begin{thm}\label{main products 1}
Let $G=G_1\times \cdots \times G_n$, $n\geqslant 1$, be a product group, where each
$G_i$ is a finitely generated free group or a surface group. Then, $\rk(\fix \phi
)\leqslant \rk(G)$ for every $\phi \in \aut(G)$ if and only if $G$ is either of
euclidean or of hyperbolic type.
\end{thm}

\begin{proof}
Proposition~\ref{exs no BH} immediately gives us the ``only if\," part.

For the ``if\," part, let us distinguish the two situations. If $G$ is of euclidean type
then, by Lemma~\ref{center-centr}, $G=NS_2^{\, \ell}\times \mathbb{Z}^p \times
(\mathbb{Z}/2\mathbb{Z})^q$ and, by Corollary~\ref{facil}, we are done.

Now assume $G$ of hyperbolic type, let $\phi \in \aut(G)$, and let us prove that
$\rk(\fix \phi )\leqslant \rk(G)$. By Proposition~\ref{autos} (and adopting the block
notation for $G$), there exist automorphisms $\phi_{i,j}\in \aut(G_i)$ and permutations
$\sigma_i \in S_{n_i}$, such that
 $$
\phi =\prod_{i=1}^m \big( \sigma_i \circ \prod_{j=1}^{n_i} \phi_{i,j} \big).
 $$
Since $\sigma_i \circ (\phi_{i,1}\times \cdots \times \phi_{i,n_i})\in \aut
(G_i^{\,n_i})$ for $i=1,\ldots ,m$, and it is clear that
 $$
\fix \phi =\fix \big( \sigma_1 \circ (\phi_{1,1}\times \cdots \times \phi_{1,n_1}) \big)
\times \cdots \times \fix \big( \sigma_m \circ (\phi_{m,1}\times \cdots \times \phi_{m,n_m})
\big),
 $$
we are reduced to prove the statement for the case $m=1$.

So, let us reduce ourselves to the situation $G=G_1^{\, n}=G_{1,1}\times \cdots \times
G_{1,n}$ (with $G_{1,i}=G_1$) and $\phi =\sigma \circ \big( \phi_1\times \cdots \times
\phi_n \big)$, for some $\sigma\in S_n$ and some $\phi_j\in \aut(G_{1,j})$, $j=1,\ldots
,n$. If $\sigma=Id$ then
 $$
\fix \phi =\fix (\phi_1 \times \cdots \times \phi_n) =\fix \phi_1 \times \cdots \times \fix
\phi_n
 $$
and so, by Theorems~\ref{BH} and~\ref{WZ thm}(i), and using Lemma~\ref{ranks}, we have
 $$
\rk(\fix \phi)\leqslant \rk(\fix \phi_1)+\cdots +\rk(\fix \phi_n)\leqslant n\, \rk(G_1)=
\rk(G_1^{\, n})=\rk (G)
 $$
and we are done.

Assume $\sigma \neq Id$ and consider its decomposition as a product of cycles with
disjoint supports, $\sigma=\tau_1\circ \cdots \circ \tau_{\ell}$ with
$\supp(\tau_1)\sqcup \cdots \sqcup \supp (\tau_{\ell})=\{ 1,\ldots ,n\}$. Then,
 $$
\phi = \tau_1\circ \cdots \circ \tau_{\ell} \circ
\big( (\prod_{i\in \supp(\tau_1)} \phi_i )\times \cdots \times (\prod_{i\in \supp(\tau_{\ell}
)}\phi_i ) \big)=
 $$
 $$
=\big( \tau_1 \circ \prod_{i\in \supp(\tau_1)} \phi_i \big)\times \cdots \times \big(
\tau_{\ell} \circ \prod_{i\in \supp(\tau_{\ell} )}\phi_i \big),
 $$
where $\big( \tau_j \circ \prod_{i\in \supp(\tau_j )}\phi_i \big) \in \aut(\prod_{i\in
\supp(\tau_j )} G_{1,i})$. By counting $\rk(\fix \phi)$ as above, we are reduced to the
case $\ell=1$ i.e., we can assume $\sigma$ to be a cycle of length $n$.

Changing the notation if necessary, we can assume $\sigma=(n,n-1,\ldots ,1)$. In this
situation, our automorphism $\phi$ has the form
 $$
\begin{array}{rcl} \phi \colon G_{1,1}\times \cdots \times G_{1,n} & \to & G_{1,1}\times
\cdots \times G_{1,n} \\ (g_1,\ldots ,g_n) & \mapsto & \sigma(\phi_1(g_1), \phi_2(g_2),
\ldots , \phi_n(g_n)) = \\ & & =(\phi_n(g_n), \phi_1(g_1), \ldots , \phi_{n-1}(g_{n-1})).
\end{array}
 $$
Note that if $(g_1,\ldots ,g_n)\in \fix \phi$ then $g_1=\phi_n(g_n),
g_2=\phi_1(g_1),\ldots ,g_n=\phi_{n-1}(g_{n-1})$ and so, $g_1 =\phi_n\cdots
\phi_1(g_1)$. Then, it is straightforward to see that
 $$
\fix \phi =\{ \big( g, \phi_1(g), \phi_2\phi_1(g),\ldots ,\phi_{n-1}\cdots \phi_1(g)\big)
\mid g\in \fix (\phi_n\cdots\phi_1)\}.
 $$
Finally, by Theorems~\ref{BH} and~\ref{WZ thm}(i), we deduce
 $$
\rk(\fix \phi )=\rk(\fix (\phi_n\cdots\phi_1))\leqslant \rk (G_1)\leqslant \rk(G_1^{\, n})
=\rk(G),
 $$
concluding the proof.
\end{proof}

The next natural question is to ask which product groups $G$ enjoy the property that
every automorphism has compressed (or inert) fixed point subgroup. Of course, they form
a subset of those product groups satisfying Bestvina--Handel inequality, described in
Theorem~\ref{main products 1}; compression is more restrictive, and inertia even more,
as can be seen in the following two results. We cannot give a full characterization but
give necessary conditions, and state two conjectures characterizing both properties.

\begin{thm}\label{main products 2}
Let $G=G_1\times \cdots \times G_n$, $n\geqslant 1$, be a product group, where each
$G_i$ is a finitely generated free group or a surface group. If $\fix \phi$ is
compressed in $G$ for every $\phi \in \aut(G)$, then $G$ must be of one of the following
forms:
\begin{itemize}
\item[(\emph{euc1})] $G=\Z^p \times (\Z/2\Z)^q$ for some $p,q\geqslant 0$; or
\item[(\emph{euc2})] $G=NS_2\times (\Z/2\Z)^q$ for some $q\geqslant 0$; or
\item[(\emph{euc3})] $G=NS_2\times \Z^p \times (\Z/2\Z)$ for some $p\geqslant 1$; or
\item[(\emph{euc4})] $G=NS_2^{\, \ell}\times \Z^p$ for some $\ell\geqslant 1$,
    $p\geqslant 0$; or
\item[] \vskip -.3cm
\item[(\emph{hyp1})] $G=F_r\times NS_3^{\, \ell}$ for some $r\geqslant 2$,
    $\ell\geqslant 0$; or
\item[(\emph{hyp2})] $G=S_g\times NS_3^{\, \ell}$ for some $g\geqslant 2$,
    $\ell\geqslant 0$; or
\item[(\emph{hyp3})] $G=NS_k\times NS_3^{\, \ell}$ for some $k\geqslant 3$,
    $\ell\geqslant 0$.
\end{itemize}
\end{thm}

\begin{proof}
Assume that $\fix \phi$ is compressed in $G$ for every $\phi \in \aut(G)$. In
particular every $\phi \in \aut(G)$ satisfies $\rk(\fix \phi)\leqslant \rk(G)$ and, by
Theorem~\ref{main products 1}, $G$ is either of euclidean type or of hyperbolic type.
In the first case, we shall prove that $G$ is specifically of form (euc1), or (euc2),
or (euc3), or (euc4). And in the second case $G$ will be of the form (hyp1), or (hyp2),
or (hyp3).

Let us assume $G$ is euclidean i.e., $G=NS_2^{\, \ell}\times \mathbb{Z}^p \times
(\mathbb{Z}/2\mathbb{Z})^q$ for some integers $\ell, p,q\geqslant 0$. The next two
paragraphs will prove that ``if $\ell\geqslant 2$ then $q=0$", and ``if $\ell,p\geqslant
1$ then $q=0,1$". By distinguishing whether $\ell=0$, or $\ell =1$, or $\ell \geqslant
2$, these two restrictions in the parameters force $G$ to fall into one of the forms
(euc1), or (euc2), or (euc3), or (euc4).

To see that $\ell\geqslant 2$ implies $q=0$, assume $\ell \geqslant 2$ and $q\geqslant
1$ and let us construct an automorphism of $G$ whose fixed subgroup is not compressed.
In this case we have $G=NS_2^{\, 2}\times (\Z/2\Z)\times G_4\times \cdots \times
G_n=\langle a,b \mid aba^{-1}b\rangle \times \langle c,d \mid cdc^{-1}d \rangle \times
\langle e \mid e^2\rangle \times G_4\times \cdots \times G_n$. Consider the map $\phi
\colon G\to G$, $a\mapsto a$, $b\mapsto be$, $c\mapsto cd$, $d\mapsto d$, $e\mapsto e$,
and fixing the rest of generators. It is straightforward to see that $\phi$ is a well
defined automorphism and, using normal forms of elements, its fixed subgroup is $\fix
\phi =\langle a, b^2, c^2, d, e\rangle \times G_4\times \cdots \times G_n =\langle
a,b^2\rangle \times \langle c^2, d\rangle \times \langle e\rangle \times G_4\times
\cdots \times G_n \simeq NS_2 \times \Z^2 \times \Z/2\Z \times G_4\times \cdots \times
G_n$. By Lemma~\ref{ranks}, $\rk(\fix \phi )=5+\rk (G_4)+\cdots +\rk (G_n )$. But $\fix
\phi$ is contained in $\langle a, bc, d, e\rangle \times G_4\times \cdots \times G_n$
(note that conjugating $bc$ by $a$ one gets $b^{-1}c$), which has rank less than or
equal to $4+\rk (G_4)+\cdots +\rk (G_n )<r(\fix \phi)$; therefore, $\fix \phi$ is not
compressed in $G$.

And, in order to see that $\ell,p\geqslant 1$ implies $q=0,1$, assume $\ell,p\geqslant
1$ and $q\geqslant 2$ and let us construct an automorphism of $G$ whose fixed subgroup
is not compressed. So, in the situation $G=NS_2\times \Z \times (\Z/2\Z)^2\times
G_5\times \cdots \times G_n=\langle a,b \mid aba^{-1}b\rangle \times \langle c \mid
\,~\rangle \times \langle d \mid d^2\rangle \times \langle e \mid e^2\rangle \times
G_5\times \cdots \times G_n$, consider the automorphism $\phi \colon G\to G$ given by
$a\mapsto a$, $b\mapsto bd$, $c\mapsto ce$, $d\mapsto d$, $e\mapsto e$, and fixing all
elements from $G_5,\ldots ,G_n$ (it is straightforward to check that this is well
defined, as well as its obvious invers). Now, it is not difficult to see that $\fix \phi
=\langle a, b^2, c^2, d, e\rangle \times G_5\times \cdots \times G_n =\langle
a,b^2\rangle \times \langle c^2\rangle \times \langle d\rangle \times \langle e\rangle
\times G_5\times \cdots \times G_n \simeq NS_2 \times \Z \times (\Z/2\Z)^2\times
G_5\times \cdots \times G_n$ which, by Lemma~\ref{ranks}, has rank $\rk(\fix \phi
)=5+\rk (G_5)+\cdots  +\rk (G_n )$; but, as in the example above, $\fix \phi$ is
contained in $\langle a, bc, d,e\rangle \times G_5\times \cdots \times G_n$, which has
rank less than or equal to $4+\rk (G_5)+\cdots  +\rk (G_n )<r(\fix \phi)$ therefore, it
is not compressed in $G$.

\medskip

Now, let us assume $G$ is hyperbolic i.e., the direct product of, possibly, several free
groups $F_r$ with $r\geqslant 2$, several orientable surface groups $S_g$ with
$g\geqslant 2$, and several non-orientable surface groups $NS_k$ with $k\geqslant 3$; it
just remains to see that, in  this case, at most one of the direct summands is not
isomorphic to $NS_3$ (so, forcing $G$ to fall into the forms (hyp1), or (hyp2), or
(hyp3)). We shall prove this by assuming two direct summands in $G$ of the form $F_r$
with $r\geqslant 2$, or $S_g$ with $g\geqslant 2$, or $NS_k$ with $k\geqslant 4$, and
constructing an automorphism whose fixed subgroup is not compressed in $G$.

The free group $F_r=\langle a_1, \ldots ,a_r \mid \, \rangle$, $r\geqslant 2$, admits
the automorphism $\phi \colon F_r \to F_r$, $a_1\mapsto a_1a_2$, $a_2\mapsto a_2$,
$a_i\mapsto a_i$ for $i=3,\ldots ,r$, whose fixed subgroup is $\fix \phi =\langle a_2,
a_1a_2a_1^{-1}, a_3, \ldots ,a_r\rangle$. Note that $\fix \phi \simeq F_r$ and so,
$\rk(\fix \phi )=r=\rk(F_r)$.

The orientable surface group $S_g=\langle a_1, b_1, a_2, b_2, \ldots ,a_g, b_g \mid
[a_1, b_1]\cdots [a_g,b_g] \rangle$, with $g\geqslant 2$, admits the automorphism $\phi
\colon S_g \to S_g$, $a_1\mapsto a_1b_1$, $b_1\mapsto b_1$, $a_2\mapsto a_2b_2$,
$b_2\mapsto b_2$, $a_i\mapsto a_i$, $b_i\mapsto b_i$ for $i=3,\ldots ,g$, whose fixed
subgroup is $\fix \phi =\langle b_1, a_1b_1a_1^{-1}, b_2, a_2b_2a_2^{-1}, a_3, b_3,
\ldots ,a_g, b_g\rangle$. Note here that, because of the defining relation for $S_g$,
$\fix \phi =\langle b_1, a_1b_1a_1^{-1}, b_2, a_3, b_3, \ldots ,a_g, b_g\rangle \simeq
F_{2g-1}$ (in fact, it is a subgroup of $\langle b_1, a_1, b_2, a_3, b_3, \ldots ,a_g,
b_g\rangle$, which is free of rank $2g-1$ by Magnus Freiheitssatz,
see~\cite[Theorem~5.1]{LS}). Observe that if $g=1$ then this fixed subgroup is cyclic
(because $a_1$ and $b_1$ commute in this case), and this is not good for the coming
argument.

Finally, the non-orientable surface group $NS_k=\langle a_1, a_2, \ldots ,a_k \mid \,
a_1^2a_2^2\cdots a_k^2\rangle$ with $k\geqslant 4$, can also be presented as $\langle a,
b, c, d, a_5,\ldots ,a_k \mid aba^{-1}bcdc^{-1}da_5^2\cdots a_k^2\rangle$ (the
isomorphism being essentially the one given in Remark~\ref{centr exc} namely, $a_1
\mapsto a$, $a_2\mapsto a^{-1}b$, $a_3 \mapsto c$, $a_4 \mapsto c^{-1}d$, $a_i\mapsto
a_i$ for $i=5,\ldots ,k$). With this new presentation, it is straightforward to see that
$NS_k$, for $k\geqslant 4$, admits the automorphism $\phi \colon NS_k \to NS_k$,
$a\mapsto ab$, $b\mapsto b$, $c\mapsto cd$, $d\mapsto d$, $a_i\mapsto a_i$ for
$i=5,\ldots ,k$, whose fixed subgroup is $\fix \phi =\langle b, aba^{-1}, d, cdc^{-1},
a_5, \ldots ,a_k\rangle$. Again note that, because of the defining relation for $NS_k$,
$\fix \phi =\langle b, aba^{-1}, d, a_5, \ldots ,a_k\rangle \simeq F_{k-1}$ (as above,
it is a subgroup of $\langle b, a, d, a_5, \ldots ,a_k\rangle$, which is free of rank
$k-1$). Observe also that $k\geqslant 4$ is crucial at this point because we need two
pairs $(a,b)$ and $(c,d)$ to do the trick (with $k=2$, $\phi$ makes perfect sense but
$aba^{-1}=b^{-1}$ and $\fix \phi$ is cyclic, which is not good for the coming argument;
and with $k=3$ one could consider the automorphism $a\mapsto ab$, $b\mapsto b$,
$c\mapsto c$, whose fixed subgroup is $\langle b, aba^{-1}, c\rangle =\langle
b,c\rangle$, not good either for the argument in the next paragraph).

\medskip

Let us retake now the main argument: suppose $G$ has at least two direct summands of the
above form i.e., $n\geqslant 2$ and both $G_1$ and $G_2$ are of the form $F_r$ with
$r\geqslant 2$, or $S_g$ with $g\geqslant 2$, or $NS_k$ with $k\geqslant 4$. Denote by
$x_1, x_2,\ldots , x_{r_1}$ and $y_1, y_2,\ldots , y_{r_2}$ the standard generators for
$G_1$ and $G_2$, respectively, where $r_1=\rk(G_1)$ and $r_2=\rk(G_2)$. By the previous
three paragraphs, we have automorphisms $\phi_1\in \aut(G_1)$ and $\phi_2\in \aut(G_2)$
such that $\fix \phi_1 =\langle x_2, x_1x_2x_1^{-1}, /x_3/, x_4, \ldots ,x_{r_1}\rangle$
and $\fix \phi_2 =\langle y_2, y_1y_2y_1^{-1}, /y_3/, y_4,\ldots ,y_{r_2}\rangle$ are
free on the listed generators, where the notation $/x_3/$ means that we omit this third
generator except in the free ambient case $G_1=F_r$. Hence, for $i=1,2$, $\rk(\fix
\phi_i)=s_i$, where $s_i=r_i$ if $G_i$ is free, and $s_i=r_i-1$ if $G_i$ is a surface
group.

Finally, consider the automorphism $\phi =\phi_1 \times \phi_2 \times Id\times \cdots
\times Id \in \aut(G)$. It happens that
 $$
\begin{array}{rl}
\fix \phi =& \hskip -.3cm \langle x_2, x_1x_2x_1^{-1}, /x_3/, \ldots ,x_{r_1}\rangle \times
\langle y_2, y_1y_2y_1^{-1}, /y_3/, \ldots ,y_{r_2}\rangle \times G_3 \times \cdots \times
G_n \\ \simeq & \hskip -.3cm F_{s_1}\times F_{s_2}\times G_3\times \cdots \times G_n.
\end{array}
 $$
On one hand, by Lemma~\ref{ranks}, we have $\rk(\fix \phi)=s_1+s_2+\rk(G_3)+\cdots
+\rk(G_n)$. On the other hand, since the $x_i$'s commute with the $y_i$'s, $\fix \phi$
is contained in the subgroup
 $$
H=\langle x_2, y_2, x_1y_1, /x_3/,\ldots ,x_{r_1}, /y_3/, \ldots ,y_{r_2}\rangle \times G_3
\times \cdots \times G_n \leqslant G,
 $$
with $\rk (H)\leqslant \rk(\fix \phi)-1$. Therefore, $\fix \phi$ is not compressed in
$G$, concluding the proof.
\end{proof}

\begin{rem}\label{euc inert}
We suspect that the implication in Theorem~\ref{main products 2} is an equivalence. To
see this, two implications are missing.

In the hyperbolic case, one should be able to see that the fixed subgroup of any
automorphism of a group $G$ of the form (hyp1), or (hyp2), or (hyp3) is compressed in
$G$. Any such automorphism is rectangular up to permutation; and we already know that
the fixed subgroup of an automorphism of any component $G_i$ is compressed in $G_i$. Now
$A_i\leqslant G_i$ being compressed in $G_i$ for $i=1,\ldots ,n$, implies that
$A=A_1\times \cdots \times A_n\leqslant G_1\times \cdots \times G_n =G$ satisfies
$r(A)\leqslant r(H)$ for every $H$ of the form $H=H_1\times \cdots \times H_n \leqslant
G$ with $A_i\leqslant H_i$. It remains to study what happens with the non-rectangular
subgroups $H$ such that $A\leqslant H\leqslant G$. (The trick used in the proof of
Theorem~\ref {main products 2} to destroy compressedness playing with such
non-rectangular subgroups does not work for $NS_3$.)

In the euclidean case, we conjecture a little more: the fixed subgroup of any
endomorphism of a group $G$ of the form (euc1), or (euc2), or (euc3), or (euc4) is inert
in $G$. This is obviously true for groups of the form (euc1); with a variation of the
argument in Corollary~\ref{facil} it is straightforward to see it for groups of the
forms (euc2) and (euc3); and it remains to study the case where $G$ is of the form
(euc4). A complete proof for this case would require a detailed analysis of the
automorphisms and endomorphisms of $G$ (note that every direct summand contributes
non-trivially to the center of $G$ and so $G$ admits automorphisms far from being
rectangular, even up to permutation). Moreover, the form of endomorphisms must play an
important role because these groups, even though looking very close to abelian, they
contain subgroups which are not compressed: in fact, following the same idea as in the
proof of Theorem~\ref{main products 2}, inside the group $G=NS_2\times \Z=\langle a,b
\mid aba^{-1}b\rangle \times \langle c \mid \, \rangle$ we have the subgroup $K=\langle
a, bc\rangle$, which contains $H=\langle a, b^2, c^2\rangle$. But $\rk(K)=2$ while
$H\simeq NS_2 \times \Z$ and so $\rk(H)=3$, therefore $H$ is not compressed in $G$.
However, the trick used above to realize $H$ as the fixed subgroup of an automorphism of
$G$ made essential use of two elements of order 2, which are not available now (and, in
fact, $H$ is not the fixed subgroup of any endomorphism of $G$, since it is easy to see
that in $G$ the only solution to the equation $x^2=b^2$ is $x=b$ so, any endomorphism
fixing $b^2$ must also fix $b$).
\end{rem}

\begin{conj}\label{conj compr}
Let $G=G_1\times \cdots \times G_n$, $n\geqslant 1$, be a product group, where each
$G_i$ is a finitely generated free group or a surface group. Then, $\fix \phi$ is
compressed in $G$ for every $\phi \in \aut(G)$ if and only if $G$ is of one of the forms
(euc1), or (euc2), or (euc3), or (euc4), or (hyp1), (hyp2), or (hyp3).
\end{conj}

As the following result expreses, inertia for the fixed subgroup of every automorphism
is an even stronger condition, at least for the hyperbolic case.

\begin{thm}\label{main products 3}
Let $G=G_1\times \cdots \times G_n$, $n\geqslant 1$, be a product group, where each
$G_i$ is a finitely generated free group or a surface group. If $\fix \phi$ is inert in
$G$ for every automorphism $\phi \in \aut(G)$, then $G$ is of one of the forms (euc1),
or (euc2), or (euc3), or (euc4), or
\begin{itemize}
\item[(\emph{hyp1'})] $G=F_r$ for some $r\geqslant 2$; or
\item[(\emph{hyp2'})] $G=S_g$ for some $g\geqslant 2$; or
\item[(\emph{hyp3'})] $G=NS_k$ for some $k\geqslant 3$.
\end{itemize}
\end{thm}

\begin{proof}
Assume that $\fix \phi$ is inert in $G$ for every $\phi \in \aut(G)$. In particular
every $\phi \in \aut(G)$ has $\fix \phi$ being compressed in $G$ and, by
Theorem~\ref{main products 2}, $G$ is of one of the forms (euc1), (euc2), (euc3),
(euc4), or (hyp1), (hyp2), or (hyp3). In the euclidean case, we are done; in the
hyperbolic case it just remains to see that $n=1$.

We shall prove this by assuming $G$ of the form (hyp1), (hyp2), or (hyp3) with
$n\geqslant 2$, and constructing an automorphisms of $G$ whose fixed subgroup is not
inert in $G$. In these three situations, we have $G=G_1\times G_2\times \cdots \times
G_n$ with $n\geqslant 2$, $G_1 =F_r, S_g$, or $NS_k$, for some $r\geqslant 2$,
$g\geqslant 2$, or $k\geqslant 3$, respectively, and $G_2= NS_3= \langle c, d, e \mid
cdc^{-1}de^2\rangle$. Consider the automorphism $\phi_2 \colon NS_3 \to NS_3$, $c\mapsto
cd$, $d\mapsto d$, $e\mapsto e$, and note that $\fix \phi_2 \cap \langle c\rangle =1$
(one can see this, for example, abelianizing). Now consider $\phi =Id\times \phi_2
\times Id \times \cdots \times Id \in \aut(G)$, and we claim that $\fix \phi =G_1\times
\fix \phi_2 \times G_3\times \cdots \times G_n$ is not inert in $G$.

Suppose $G_1 =F_r =\langle a_1, \ldots ,a_r \mid \, \rangle$ with $r\geqslant 2$. Take
$K=\langle ca_1, a_2, \ldots ,a_r\rangle \leqslant G$, and consider the projection
$\pi\colon F_r \twoheadrightarrow \Z$, $w\mapsto |w|_1$. It is easy to see that $\fix
\phi \cap K=\fix (Id\times \phi_2)\cap K =(F_r\times \fix \phi_2 )\cap K$ equals $\ker
\pi$, which is a normal subgroup of infinite index in $F_r$, so infinitely generated.
In fact, every element in $\ker \pi$ is a word $w(a_1, a_2, \ldots ,a_r)$ with
$|w|_1=0$ and so, $w(a_1, a_2, \ldots ,a_r)=c^{|w|_1}w(a_1, a_2, \ldots ,a_r)=w(ca_1,
a_2, \ldots ,a_r)\in K$ since $c$ commutes with all the $a_i$'s; and conversely, if
$w(a_1, a_2, \ldots ,a_r)v=w'(ca_1, a_2, \ldots ,a_r)$ for some $v\in \fix \phi_2$,
then $w(a_1, a_2, \ldots ,a_r)v=w'(a_1, a_2, \ldots ,a_r)c^{|w'|_1}$ which implies
$w=w'$, $v=1$, and $|w'|_1=0$ therefore, $w(a_1, a_2, \ldots ,a_r)v\in \ker \pi$.

Suppose now that $G_1=S_g=\langle a_1,b_1,\ldots,a_g,b_g \mid [a_1, b_1]\cdots [a_g,
b_g]\rangle$ with $g\geqslant 2$. Take $K=\langle ca_1, b_1, a_2, b_2, \ldots ,a_g,
b_g\rangle \leqslant G$, and consider the projection $\pi\colon S_g \twoheadrightarrow
\Z$, $w\mapsto |w|_1$. The same argument as above shows that $\fix \phi \cap K=\ker
\pi$, which in this case is infinitely generated as well, by the argument given in the
proof of Proposition~\ref{exs no BH} (case $G_2=S_g$).

Finally, suppose $G_1=NS_k =\langle a, b, a_3,\ldots,a_k \mid aba^{-1}ba_3^2\cdots
a_k^2\rangle$ with $k\geqslant 3$, and consider the projection $\pi\colon NS_k
\twoheadrightarrow \Z$, $w\mapsto |w|_1$ (which coincides with that in the proof of
Proposition~\ref{exs no BH} case $G_2=NS_k$, where it is expressed with respect to the
other usual presentation of $NS_k$). The exact same argument as in the previous case
shows that $\fix \phi \cap K=\ker \pi$, which is again infinitely generated.
\end{proof}

A positive solution to Conjectures~\ref{conj} and~\ref{conj surfaces}, and to that
suggested in Remark~\ref{euc inert}, would give a positive solution to the following
one.

\begin{conj}\label{conj inert}
Let $G=G_1\times \cdots \times G_n$, $n\geqslant 1$, be a product group, where each
$G_i$ is a finitely generated free group or a surface group. Then, the following are
equivalent:
\begin{itemize}
\item[(a)] every $\phi \in \edo(G)$ satisfies that $\fix \phi$ is inert in $G$,
\item[(b)] every $\phi \in \aut(G)$ satisfies that $\fix \phi$ is inert in $G$,
\item[(c)] $G$ is of the form (euc1), or (euc2), or (euc3), or (euc4), or (hyp1'), or
    (hyp2'), or (hyp3').
\end{itemize}
\end{conj}

\begin{cor}
For product groups, the ``compressed-inert" conjecture is false even for fixed
subgroups of automorphisms i.e., there exists a product group $G$ and $\phi \in
\aut(G)$ such that $\fix \phi$ is compressed in $G$ but not inert in $G$.
\end{cor}

\begin{proof}
Conjectures~\ref{conj compr} and~\ref{conj inert} already suggest that such $G$ and
$\phi$ do exist. The easiest example is $G=F_2\times \Z =\langle a,b \mid \, \rangle
\times \langle c \mid \, \rangle$ and $\phi \colon G\to G$, $a\mapsto a$, $b\mapsto b$,
$c\mapsto c^{-1}$. Clearly, $\fix \phi =\langle a,b\rangle$ is compressed. But $\fix
\phi \cap \langle ca,b\rangle$ is the normal closure of $b$ in $F_2$, which is
infinitely generated. Hence, $\fix \phi$ is not inert in $G$.
\end{proof}

\bigskip

\noindent\textbf{Acknowledgements.} We thank Laurent Gwena\"el for an interesting
conversation leading us to the idea of considering product groups. The first named
author is partially supported by NSFC (No. 11201364). The second one acknowledges
partial support from the Spanish Government through grant number MTM2011-25955. The
third named author is funded by China Scholarship Council, partially supported by NSFC
(No. 11271276), and thanks the hospitality of the Universitat Polit\`ecnica de Catalunya
in hosting him as a postdoc during the academic course 2014--2015, while this research
was conducted.



\begin{thebibliography}{99}
\bibitem{B} G. Bergman, ``Supports of derivarions, free factorizations and ranks
    of fixed subgroups in free groups", \emph{Trans. Amer. Math. Soc.} \textbf{351} (1999),
    1531--1550.

\bibitem{BH} M. Bestvina and M. Handel, ``Train tracks and automorphisms of free
    groups", \emph{Annals of Mathematics} \textbf{135} (1992), 1--51.

\bibitem{DeV} J. Delgado, E. Ventura, ``Algorithmic problems for free-abelian times
    free groups", \emph{Journal of Algebra} \textbf{391} (2013), 256--283.

\bibitem{D} W. Dicks, ``Simplified Mineyev", preprint,
    http://mat.uab.cat/~dicks/SimplifiedMineyev.pdf.

\bibitem{DV} W. Dicks and E. Ventura, ``The group fixed by a family of injective
    endomorphisms of a free group", \emph{Contemporary Mathematics} \textbf{195}, AMS,
    Providence (1996).

\bibitem{FM} B. Farb and D. Margalit, \emph{A primer on mapping class groups}, Princeton
    University Press (2012).

\bibitem{F} J. Friedman, ``Sheaves on graphs, their homological invariants,
    and a proof of the Hanna Neumann conjecture: with an appendix by Warren Dicks",
    \emph{Mem. Amer. Math. Soc.} \textbf{233} (2014), no. 1100. xii+106 pp.

\bibitem{Hem} J. Hempel, ``Residual finiteness of surface groups", \emph{Proc. Amer.
    Math. Soc.} \textbf{32} (1972), 323.

\bibitem{Hir} R. Hirshon, ``On cancellation in groups", \emph{Amer. Math. Monthly}
    \textbf{76} (1969), 1037--1039.

\bibitem{JWZ}B. Jiang, S. Wang and Q. Zhang, ``Bounds for fixed points and
    fixed subgroups on surfaces and graphs", \emph{Alg. Geom. Topology}, \textbf{11} (2011),
    2297--2318.

\bibitem{LS}R. Lyndon, P. Schupp, \emph{Combinatorial group theory}, Ergebnisse der
    Mathematik und ihrer Grenzgebiete 89, Springer, Berlin (1977).

\bibitem{MV} A. Martino and E. Ventura, ``Fixed subgroups are compressed in free
    groups", \emph{Communications in Algebra} \textbf{32} (2004), no. 10, 3921--3935.

\bibitem{M1} I. Mineyev, ``The topology and analysis of the Hanna Neumann
    Conjecture", \emph{J. Topol. Anal.} \textbf{3} (2011), 307--376.

\bibitem{M} I. Mineyev, ``Submultiplicativity and the Hanna Neumann Conjecture",
    \emph{Annals of Mathematics} \textbf{175} (2012), 393--414.

\bibitem{N1} H. Neumann, ``On the intersection of finitely generated free
    groups", \emph{Publ. Math. Debrecen} \textbf{4} (1956), 186--189.

\bibitem{N2} H. Neumann, ``On the intersection of finitely generated free
    groups, Addendum", \emph{Publ. Math. Debrecen} \textbf{5} (1957), 128.

\bibitem{T} G. Tard\"os, ``On the intersection of subgroups of a free group",
    \emph{Invent.
    Math.} \textbf{108} (1992), 29-36.

\bibitem{V} E. Ventura, ``Fixed subgroups in free groups: a survey", \emph{Contemp.
    Math.} \textbf{296} (2002), 231-255.

\bibitem{WZ} J. Wu and Q. Zhang, ``The group fixed by a family of endomorphisms of
    a surface group", \emph{J. Algebra} \textbf{417} (2014), 412--432.

\bibitem{Z} Q. Zhang, ``Bounds for fixed points on Seifert manifolds", \emph{Topology
    and its Applications} \textbf{159} (2012), 3263--3273.
\end{thebibliography}
\end{document}